\documentclass[graybox]{svmult}

\usepackage{type1cm}        
\usepackage{makeidx}         
\usepackage{graphicx}        
\usepackage{multicol}        
\usepackage[bottom]{footmisc}

\usepackage{newtxtext}       %
\usepackage[varvw]{newtxmath}       

\def\be{\begin{eqnarray}}
\def\ee{\end{eqnarray}}
\def\ben{\begin{eqnarray*}}
\def\een{\end{eqnarray*}}
\newcommand{\RL}{{\mathbb R}}
\newcommand{\Phatn}{\mbox{$\hat{P}_n$}}
\newcommand{\Pp} {
\mathbb{P}
}

\newcommand{\IND}{{\mathbb I}}
\newcommand{\What}{\widehat{W}}

\newcommand{\leqb}{\mbox{$ \;\stackrel{(b)}{\leq}\; $}}
\newcommand{\leqc}{\mbox{$ \;\stackrel{(c)}{\leq}\; $}}
\newcommand{\leqd}{\mbox{$ \;\stackrel{(d)}{\leq}\; $}}

\newcommand{\eqa}{\mbox{$ \;\stackrel{(a)}{=}\; $}}

\makeindex             

\begin{document}

\title*{Information in probability:\\ Another
information-theoretic proof of a finite de Finetti theorem}
\titlerunning{Information in probability} 
\author{Lampros Gavalakis and Ioannis Kontoyiannis}
\institute{Lampros Gavalakis \at 
Department of Engineering, University of Cambridge, Trumpington Street,
Cambridge CB2 1PZ, U.K. \email{lg560@cam.ac.uk}
\and Ioannis Kontoyiannis \at 
Statistical Laboratory, DPMMS, University of Cambridge, 
Centre for Mathematical Sciences, Wilberforce Road,
Cambridge CB3 0WB, U.K. \email{yiannis@maths.cam.ac.uk}}
%
%
\maketitle

\abstract{We recall some of the history of the information-theoretic approach 
to deriving core results in probability theory and 
indicate parts of the recent resurgence of interest in this area 
with current progress along several interesting directions. Then
we give a new information-theoretic proof 
of a finite version of de Finetti's classical representation theorem
for finite-valued random variables.
We derive an upper bound on the relative entropy between 
the distribution of the first $k$ in a sequence of $n$ exchangeable 
random variables, and an appropriate mixture over product distributions.
The mixing measure is characterised as the law of the 
empirical measure of the original sequence,
and de Finetti's result is recovered 
as a corollary.
The proof is nicely motivated by the Gibbs conditioning principle
in connection with statistical mechanics, and it follows along
an appealing sequence of steps. The technical estimates required
for these steps are obtained via the use of a collection
of combinatorial tools known within
information theory as  `the method of types.'}

\section{Entropy and information in probability}

Shannon's landmark 1948 paper~\cite{shannon:48} founded the field 
of information theory and ignited the fuse that led to much of the 
subsequent explosive development of communications theory and 
engineering in the 20th century. At the same time, it also led to a
wave of applications of information theory to numerous other
branches of science. Of those, some, e.g.\ those in bioinformatics
and neuroscience, were successful, while some others,
despite Shannon's ``bandwagon'' warning~\cite{shannon:56},
much less so.

Within mathematics, information-theoretic ideas have had a major
impact along several directions, perhaps most notably (although certainly
not exclusively) in connection with probability theory. For our
present purposes, the most relevant line of work is based on the 
idea of utilising information-theoretic tools and ideas in order
to prove core probabilistic results. Over the past 55 years,
a great number of such proofs have appeared.
These are often accompanied by new interpretations
and rich intuition, thus providing new ways of understanding 
why fundamental probabilistic theorems are true,
and sometimes also giving stronger versions of the original results. 

In the rest of this introduction we describe some of the
main landmarks along this path, and we indicate 
directions of current and likely near-future activity.
This brief survey is necessarily incomplete and 
biased, due to our own subjective taste and bounded knowledge.
Then in Section~\ref{sectiondefinetti} we state and prove
a new finite version of de Finetti's classical
representation theorem for finite-valued
exchangeable random variables.

The first appearance of information-theoretic ideas in the
proof of a genuinely probabilistic result was
in 1958, when H{\'a}jek~\cite{hajek:58a,hajek:58b} proved that the
laws $\mu$ and $\nu$ of any two Gaussian processes are 
either absolutely continuous with respect to each other, or singular. 
H{\'a}jek exploited the implications of $D(\mu\|\nu) + D(\nu\|\mu)$ 
being finite or infinite, where $D(\mu\|\nu)$ denotes 
the {\em relative entropy}
or {\em Kullback-Leibler divergence} 
between $\mu$ and $\nu$,
\be
D(\mu\|\nu):=\begin{cases}
\int \log\frac{d\mu}{d\nu} \,d\mu,&\mbox{if}\;\frac{d\mu}{d\nu}\;\mbox{exists}\\
+\infty,&\mbox{otherwise.}
\end{cases}
\label{eq:Ddef}
\ee
[Throughout, log denotes the natural logarithm.]
In the same year, Kolmogorov~\cite{kolmogorov:58} introduced
\textit{entropy} in ergodic theory. He provided a way to calculate the 
entropy of a transformation to conclude that Bernoulli shifts of different 
entropies are not metrically isomorphic.
The importance of entropy in ergodic theory was also highlighted more 
than a decade later, 
when Ornstein~\cite{ornstein:70a,ornstein:70b,ornstein:70c} proved 
that Bernoulli 
shifts with the same entropy are necessarily isomorphic.

The following year, 1959, Linnik~\cite{linnik:59} gave an information-theoretic
proof of the central limit theorem (CLT), showing that the law of the 
standardised sum
$S_n=(1/\sigma\sqrt{n})\sum_{i=1}^nX_i$
of $n$ independent and identically 
distributed
(i.i.d.) 
random variables $X_1,\ldots,X_n$ 
with variance $\sigma^2$ (or, more generally, 
of independent random variables satisfying the Lindeberg condition)
converges in distribution to a Gaussian. Linnik's connection between
the CLT and information-theoretic ideas was the first in a long
series of works, along a path that remains active until today. 
Indeed, in a sequence of papers, including works by Shimizu~\cite{shimizu:75},
Brown~\cite{brown:82}, 
Barron~\cite{barron:clt},
Johnson~\cite{johnson:00},
Artstein et al.~\cite{artstein:04},
Tulino and Verd{\'u}~\cite{tulino:06},
and Madiman and Barron~\cite{madiman:07},
it was shown that the {\em differential entropy}
$h(S_n)$ of the standardised sums
in fact {\em increases} with $n$,
and its limiting value is the
entropy $h(Z)$ of the standard
Gaussian $Z\sim N(0,1)$. 
This monotonic convergence in combination 
with the fact that the Gaussian has maximum
entropy among all random variables with 
variance~$\sigma^2$, presents an appealing
analogy between the CLT and the second
law of thermodynamics.

In the above discussion, the differential 
entropy of a continuous random variable
$X$ with density $f$ is given by
$h(X)=h(f)=-\int f\log f$, and the relative
entropy between two probability measures
$\mu,\nu$
on $\RL$ with densities $f,g$ is
$D(\mu\|\nu)=D(f\|g)=\int f\log (f/g)$.
A simple computation shows that the
``entropic CLT'' just described can 
equivalently be stated as,
$D(f_n\|\phi)\downarrow 0$ as $n\to\infty$,
where $f_n$ is the density of $S_n$
and $\phi$ the standard normal density.
This convergence in the sense of relative 
entropy implies convergence in $L^1$:
{\em Pinsker's inequality},
established by
Csisz{\'a}r~\cite{csiszar:67}, Kullback~\cite{kullback:67} 
and Kemperman~\cite{kemperman:69},
states that:
\begin{equation} \label{pinsker}
D(\mu\|\nu) \geq \frac{1}{2\log{2}}\|\mu-\nu\|_1^2.
\end{equation}

Subsequent work along these lines includes
Carlen and Soffer's dynamical systems approach~\cite{carlen:91}, 
Johnson's convergence to Haar measure
on compact groups~\cite{johnson:00b},
Johnson and Barron's rates of convergence in the
entropic CLT~\cite{johnson-barron:04},
Bubeck and Ganguly's
entropic CLT for Wishart random matrices~\cite{bubeck:18},
and, most recently, an information-theoretic
CLT for discrete random variables~\cite{gavalakis-arxiv:clt}.

A year after Linnik's paper made
the first entropy-CLT connection, in 1960, 
R{\'e}nyi~\cite{renyi:61} examined
the convergence of Markov chains to equilibrium from
an information-theoretic point of view, thus initiating
another path of information-theoretic investigation
in probability. R{\'e}nyi showed that the 
relative entropy $D(P_n\|\pi)$ between the time-$n$ distribution
$P_n$ of a finite-state chain with an all-positive transition
matrix and its unique invariant distribution $\pi$,
decreases to zero as $n\to\infty$.
Similar and slightly more general results were established 
independently by Csisz{\'a}r~\cite{csiszar:63} in
1963, who also employed R{\'e}nyi's notion of {\em $f$-divergence},
an important generalisation of relative entropy.
In the same year, Kendall~\cite{kendall:63} extended 
R{\'e}nyi's techniques and results, to include
certain countable state space chains. 
A significant advance came with Fritz's 1973 work~\cite{fritz:73}, 
where he studied the 
asymptotic behavior of reversible Markov kernels
and established their weak convergence
to equilibrium. Barron in 2000~\cite{barron:isit00}
extended Fritz's result to convergence in relative 
entropy, 
and in 2009 Harremo{\"e}s and Holst~\cite{harremoes:09} 
used ideas related to information projections 
to further extend and generalise those
earlier results.

The problem of Poisson approximation
and convergence 
was first examined through the lens
of information theory
around 20 years ago, 
leading to a development 
analogous to that of the entropic CLT.
Harremo\"{e}s in 2001 identified the Poisson as the maximum
entropy distribution among all laws that arise from sums
of independent Bernoulli random variables with a fixed 
mean~\cite{harremoes:01}. This characterisation was extended in 
2007 by Johnson~\cite{johnson:07} to the class of ultra 
log-concave laws on the nonnegative integers.
Meanwhile, in 2005 Kontoyiannis et al.~\cite{konto-H-J:05} 
derived convergence results and nonasymptotic
Poisson approximation bounds using entropy-theoretic 
methods. Interestingly,
some of those results were based, in part, on
a discrete modified logarithmic Sobolev inequality
for the entropy established by 
Bobkov and Ledoux~\cite{bobkov-ledoux:98}.
In a related direction, 
Harremo\"{e}s et al.~\cite{harremoesJK:07,konto-H-J:10}
obtained Poisson approximation results under the 
thinning operation.

A similar program was carried out in the case of compound Poisson
approximation. Compound Poisson laws on the integers
were again given a natural
maximum-entropy interpretation by Johnson et al.~\cite{jKm:13}
and Yu~\cite{yu:09}, and compound Poisson approximation
bounds and convergence results were established
via information-theoretic techniques by
Madiman et al.~\cite{madimanJK:07}
and Barbour et al.~\cite{Kcompound:10}.
Interestingly, in some cases the resulting nonasymptotic bounds 
give the best results to date.

\smallskip

\noindent
{\bf The method of types and large deviations. }
Suppose $\{X_n\}$ are i.i.d.\ random variables
with common probability mass function (PMF) $Q$
on a finite alphabet $A$ of size $m=|A|$.
The {\em type} $\Phatn$ of a string $x_1^n=(x_1,\ldots,x_n)\in A^n$
is simply the empirical PMF induced by $x_1^n$ on $A$.
Let ${\cal P}_n$ denote the collection of all {\em $n$-types}
on $A$, namely, all PMFs that arise as types of 
strings of length~$n$. Then, e.g., we have the obvious 
bound, 
\be
|{\cal P}_n|\leq (n+1)^m,
\label{eq:ntypes}
\ee
and direct computation also shows that,
for any $x_1^n\in A^n$,
\be
Q^n(x_1^n)=e^{-n[H(\hat{P}_n)+D(\hat{P}_n\|Q)]}.
\label{eq:exact}
\ee
Here, $H(P):=-\sum_{x\in A} P(x)\log P(x)$ is the 
(discrete Shannon) {\em entropy} of a PMF $P$ on $A$,
and the definition~(\ref{eq:Ddef}) of
the relative entropy $D(P\|Q)$ between two PMFs $P,Q$ on the
same discrete alphabet becomes
$D(P\|Q)=\sum_{x\in A}P(x)\log [P(x)/Q(x)]$.
Slightly more involved calculations
lead to interesting and useful bounds. For example,
for an $n$-type $P$,
let $T(P)$ denote the {\em type class} of~$P$, 
consisting of all $x_1^n\in A^n$ with type $P$.
Then the cardinality and probability of $T(P)$
satisfy,
\be
(n+1)^{-m}
e^{nH(P)}
&\leq& 
	|T(P)| \;\leq\;e^{nH(P)}
	\label{eq:sizesimple}\\
(n+1)^{-m}
e^{-nD(P\|Q)}
&\leq& 
	Q^n(T(P))\; \leq\;e^{-nD(P\|Q)}.
	\label{eq:probsimple}
\ee

The method of types is a collection of combinatorial
estimates for probabilities associated with discrete i.i.d.\
random variables and memoryless channels,
of which the examples in (\ref{eq:ntypes})--(\ref{eq:probsimple}) 
above are the starting point.
Based in part on preliminary ideas of Wolfowitz~\cite{wolfowitz:61book},
the method of types was
fully developed in 1981 by 
Csisz{\'a}r and K{\"o}rner~\cite{csiszar:book}. 
As described in Csisz{\'a}r's review~\cite{csiszar:98},
the method of types has been employed very
widely and with great success
in numerous information-theoretic problems arising 
from different communication-theoretic scenarios. 

Based in part on the method of types, 
and also building on ideas from related
work by Groenebook et at.~\cite{GOR:79},
Csisz{\'a}r was able to
establish a series of important results in large deviations.
In 1975~\cite{csiszar:75} he identified
the exponent in Sanov's theorem~\cite{sanov:57} 
as an extremum of relative entropies, and 
in 1984~\cite{csiszar:84} he proved a general,
strong version of Sanov's theorem,
by a combination of the method of types,
discretisation arguments, and a general
Pythagorean inequality for the relative 
entropy established by Tops{\o}e~\cite{topsoe:79}.
He also gave a simpler proof along the same
lines in his
2006 paper~\cite{csiszar:06}.

Moreover, in the same paper~\cite{csiszar:84} 
Csisz{\'a}r
established a version
of the {\em Gibbs conditioning principle} 
(also know as the {\em conditional limit theorem})
using the same tools.
This was further
extended by Csisz{\'a}r et al.\ in 1987~\cite{csiszar-cover-choi:87}
to the case of Markov conditioning, 
and by Algoet et al.\ in 1992~\cite{algoet:92}
to Markov types.

The method of types and the Gibbs conditioning principle
will both play an important role in our proof of the finite
de Finetti theorem in Section~\ref{sectiondefinetti}.

\smallskip

\noindent
{\bf Exchangeability. }
Suppose $\{X_n\}$ are i.i.d.\ random variables,
and let $\mathcal{E}$ denote the {\em exchangeable}
$\sigma$-algebra, that is, the 
sub-$\sigma$-algebra 
of $\sigma(\{X_n\})$ that consisting of those
events that are invariant under finite permutations 
of the indices in the sequence $\{X_n\}$.
In 2000,
O'Connell~\cite{oconnell:TR} gave a beautiful,
elementary 
information-theoretic proof the Hewitt-Savage 0-1 
law~\cite{hewitt-savage:55}: $\mathcal{E}$ 
is trivial, in that all events in ${\cal E}$ have
probability either zero ore one.

Another aspect of exchangeability comes up in connection
with de Finetti's theorem.
Let $\{X_n\}$ be an exchangeable sequence of random
variables with values in the same finite alphabet $A$.
Here, exchangeability means that, for any $n$ and any
permutation $\pi$ on $\{1,2,\dots,n\}$, the 
distribution of the random variables
$(X_{\pi(1)},X_{\pi(2)},\ldots,X_{\pi(n)})$
is the same as that of 
$(X_1,X_2,\ldots,X_n)$.
De Finetti's theorem 
\cite{definetti:31,definetti:37} 
states that $\{X_n\}$ is exchangeable
if and only if it is a mixture of 
i.i.d.\ sequences, that is, 
if and only if there is a measure 
$\bar{\mu}$ on the simplex ${\cal P}$ 
of probability distributions on $A$,
such that, for any $k\geq 1$ and any
$x_1^k=(x_1,\ldots,x_k)\in A^k$,
\be
\mathbb{P}(X_1^k=x_1^k)
=M_{\bar{\mu},k}(x_1^k):=\int_{\cal P} Q^k(x_1^k)d\bar{\mu}(Q).
\label{eq:deF}
\ee

De Finetti's theorem plays an 
important role in the foundations of subjective
probability and Bayesian statistics, see, e.g.,
the discussions in~\cite{diaconis:77,bayarri:04}.
But arguments about its practical relevance are
limited by the fact that, as is 
well known~\cite{diaconis-freedman:80b},
the representation~(\ref{eq:deF}) fails
in general if it is only assumed that 
a finite collection of random variables
$(X_1,\ldots,X_n)$ is exchangeable for 
some fixed $n$.
Nevertheless, approximate versions of~(\ref{eq:deF})
remain valid in this 
case~\cite{diaconis:77,diaconis-freedman:80b}.
Such a `finite' version of de Finetti's theorem
for binary random variables
was recently established in~\cite{gavalakis:21},
using information-theoretic ideas:
It was shown that there is a mixing
measure $\mu$ on ${\cal P}$ such that,
for any $1\leq k\leq n$,
the distribution $P_k$ of $(X_1,\ldots,X_k)$
is close to $M_{\mu,k}$ in the precise
sense that:
\begin{equation} 
\label{firstdefinetti}
D(P_k \| M_{\mu,k}) \leq \frac{5k^2\log{n}}{n-k}.
\end{equation}
A different information-theoretic proof of a different
finite version of de Finetti's theorem is given
in Section~\ref{sectiondefinetti}.

\smallskip

\noindent
{\bf Further connections. }
There are numerous other directions
along which information-theoretic methods have been 
employed to establish either known or new probabilistic
results. We briefly mention only a few more from the long list 
of relevant works, some of which go beyond probability 
theory.  
The interested reader may also consult
Barron's reviews of information-theoretic proofs 
and connections with statistics and 
learning~\cite{barron:97,barron:isit00},
Csisz{\'a}r's 
review of information-theoretic 
methods in probability~\cite{csiszar:97},
and Johnson's text~\cite{johnson:book}.

A natural and powerful connection has been drawn between 
information theory and concentration of measure inequalities,
through what has come to be known
as the {\em entropy method}. Often attributed 
Herbst~\cite{davies-simon:84},
the entropy method was primarily developed by 
Ledoux~\cite{ledoux:96,ledoux:97,ledoux:book}.
Marton's 1996 work~\cite{marton:96} 
had an early and 
significant influence in this direction as well. Entropy 
also appears naturally
in connection with related work on
transportation theory~\cite{bobkov:99,villani:book}.
Book-length accounts of measure concentration
and related inequalities, including the entropy
method, are given in~\cite{boucheron:book} and~\cite{raginsky:book}.

A fascinating and multifaceted series of connections
between information-theoretic ideas and functional
inequalities started with Shannon's entropy power inequality
(EPI), stated in his original 1948 
paper~\cite{shannon:48} and later proved
by Stam~\cite{stam:59} and Blachman~\cite{blachman:65}.
Much of the relevant literature up to 1991 is summarised in 
Dembo et al.'s review~\cite{dembo-cover-thomas:91},
including the connection with Gross' celebrated
Gaussian logarithmic Sobolev inequality~\cite{gross:75}.
This paper also
contains an early discussion of the strong
ties
between entropy inequalities and
high-dimensional convex geometry,
starting with 
Costa and Cover’s 1984 observation~\cite{costa-cover:84} 
that the 
Brunn-Minkowski inequality can be viewed as a special
case of a generalised EPI.

Building on the technical ideas of Stam and Blachman,
Bakry and {\'E}mery in a very 
influential 1985 paper~\cite{bakry-emery:85}
derived an important representation of the 
derivative of the relative entropy
$D(P_t\|Q_t)$ of the time-$t$ distributions
of a diffusion with different initial conditions.
Under appropriate assumptions, strong connections
were established with logarithmic Sobolev
inequalities, generalising the earlier connection
between the EPI and Gross' Gaussian logarithmic
Sobolev inequality, and facilitating the study
of the long-term behaviour of the underlying
diffusion. An important observation,
independently re-discovered by Barron \cite{barron:97},
is that this derivative can be expressed as a
``relative'' Fisher information, which also admits
an interpretation as a minimum mean squared error.
This interpretation had been promoted earlier in 
work by Brown, see e.g.~\cite{brown:86}, 
and it was re-framed in more information theoretic
terms by Guo et al. in 2005~\cite{guo:05},
leading to a variety of subsequent developments.

More recently, a remarkable equivalence between the subadditivity 
property of entropy and the classical Brascamp-Lieb inequality 
was pointed out 
by Carlen and Cordero-Erausquin~\cite{carlen:09},
and a unified information-theoretic treatment
was given by Liu et al.~\cite{courtade:16}.
In yet another direction, Tao in 2010~\cite{tao:10} 
developed a series of discrete entropy inequalities
motivated by sumset and inverse sumset bounds
in additive combinatorics, also leading to a discrete
version of the EPI. More recent work in this direction
includes~\cite{KM:14,KM:16}.

Finally, we mention that a natural analog of the entropy
in free probability was introduced in a series of papers by
Voiculescu~\cite{voiculescu:I,voiculescu:II,voiculescu:III,voiculescu:IV},
where several properties of the free entropy are established,
including a free version of the EPI.
In related work, convergence results to maximum free entropy 
distributions is considered by Johnson in~\cite[Chapter 8]{johnson:book}.

\section{Information-theoretic proof of a finite 
de Finetti theorem} 
\label{sectiondefinetti}

Suppose $X_1^n=(X_1,\ldots,X_n)$, for some fixed $n$,
are exchangeable, discrete random variables,
with values in a finite alphabet
$A$ of $m=|A|$ elements. Let $\hat{P}_{X_1^n}$ denote the 
(random) type of $X_1^n$, and let the measure $\mu=\mu_n$ denote 
the law of $\hat{P}_{X_1^n}$ on the probability simplex ${\cal P}$.
In this section we provide an information-theoretic proof of
the following:

\begin{theorem}[Finite de Finetti theorem]
\label{thm}
For any $1\leq k\leq n$,
let $P_k$ denote the distribution of 
$X_1^k=(X_1,\ldots,X_k)$ and
$M_{\mu,k}$ denote the mixture-of-i.i.d.s:
$$
M_{\mu,k}(x_1^k)=\int_{\cal P} Q^k(x_1^k)d\mu(Q),
\qquad x_1^k\in A^k.$$
For any $1\leq k\leq (n/100)^{1/3}$, we have,
\begin{equation} 
\label{finalbound}
D(P_k\|M_{\mu,k}) \leq \epsilon(n,k):=
2\delta 
+ke^{- \frac{n}{k}\delta}
\Bigl(\frac{n}{k}+1\Bigr)^{2m^k}\log n,
\end{equation}
with
$\alpha=\alpha_{n,k}= \big[\frac{2k}{\sqrt{n}}\big(\frac{1+2k}{\sqrt{n}}
+ 1
\big)\big]^{1/2}$
and
$\delta=\delta_{n,k} =\alpha
\log(m^k/\alpha)$.
\end{theorem}

\noindent
Before giving the proof of the theorem,
some remarks are in order:
\begin{enumerate}
\item
It can be seen from \eqref{finalbound} that,
if $k$ stays bounded as $n\to\infty$,
then:
$$\epsilon(n,k) 
= O(\delta_{n,k}) 
= O\Bigl(\Bigl(\frac{k}{\sqrt{n}}\Bigr)^{1/2}\log{\frac{n}{k}}\Bigr)\to 0.$$
Moreover, in order for $\epsilon(n,k)$ to vanish, 
$k$ can grow at most logarithmically with $n$. 
This is, at least asymptotically, weaker than the 
bound~(\ref{firstdefinetti}) given in~\cite{gavalakis:21}
for the binary case $m=2$. What's more,
the proof of~(\ref{finalbound}) given below
is longer and more involved that the
corresponding proof of~(\ref{firstdefinetti})
in \cite{gavalakis:21}. So why bother?
The reason is that the proof given here
follows a completely different
information-theoretic path than that in~\cite{gavalakis:21},
and that path consists of an appealing sequence of steps
making interesting connections. So we first 
present a heuristic outline,
and then give the actual proof.
In fact, as will be seen from the proof
(especially Lemma~\ref{lem:1}), it is easy 
to improve the bound $\epsilon(n,k)$,
but our purpose here is to illustrate
the ideas rather than to obtain optimal
results.
\item
We have cheated slightly in the statement of the theorem,
in that the proof below is only given for the case when
$n$ is a multiple of $k$. However, this is only a minor
technical inconvenience; for example,
we can replace $n$ with an integer multiple of $k$ 
which is no less than $n-k$, leading to the same 
bound with $\epsilon(n-k,k)$
in place of $\epsilon(n,k)$.
\item
Fe Finetti's original theorem~(\ref{eq:deF})
easily follows from~(\ref{finalbound}) by an application of Pinsker's 
inequality \eqref{pinsker} and a standard 
weak convergence argument. 
\end{enumerate}

\noindent
{\sc Heuristic proof of de Finetti's theorem~(\ref{eq:deF}). }

{\em Step 1}: Since the sequence $\{X_n\}$ is exchangeable it is 
also stationary, therefore, by the ergodic theorem
$\hat{P}_{X_1^n}$ converges as $n\to\infty$
a.s. to a (random) $P$ on $A$. Let $\bar{\mu}$ denote
the law of $P$, and let
$\{Y_n\}$ be i.i.d.\ random variables 
uniformly distributed on $A$. Then, by exchangeability, we 
clearly have for any $n$, any $k\leq n$, 
any $n$-type $Q_n$, and any $a_1^k\in A^k$,
\begin{eqnarray}
\Pp(X_1^k =a_1^k|\hat{P}_{X_1^n}  = Q_n)
  =   \Pp(Y_1^k =a_1^k|\hat{P}_{Y_1^n} = Q_n).
\label{eq:corr}
\end{eqnarray}

{\em Step 2}: 
Choose and fix any one
of the almost all realisations $\{Q_n\}$ along which $\hat{P}_{X_1^n}$
converges to some $Q$ as $n\to\infty$. By (\ref{eq:corr})
and symmetry we have,
$$\Pp(X_1=a|\hat{P}_{X_1^n}  = Q_n)
=
	{\mathbb E}\Big(\IND_{\{Y_1=a\}}\Big|\hat{P}_{Y_1^n}=Q_n\Big)
=
	{\mathbb E}\Big(\frac{1}{n} \sum _{i=1}^n \IND_{\{Y_i=a\}}\Big|
		\hat{P}_{Y_1^n}=Q_n\Big),
$$
so that,
$$\Pp(X_1=a|\hat{P}_{X_1^n}  = Q_n)
=
 	{\mathbb E}\Big(\hat{P}_{Y_1^n}(a)\Big|\hat{P}_{Y_1^n}=Q_n\Big)\\
=
	 Q_n(a),$$
for any $a\in A$,
and letting $n\to\infty$ yields,
\be
\lim_{n \to \infty} \Pp(X_1=a|\hat{P}_{Y_1^n}=Q_n)=Q(a).
\label{eq:step2}
\ee

{\em Step 3}: 
Next we generalise (\ref{eq:step2}) to blocks
of random variables.
As before, choose and fix any one
of the almost all realisations $\{Q_n\}$ of 
the random $\hat{P}_{X_1^n}$ such that
$Q_n\to $ some $Q$ as $n\to\infty$. 
Define a new sequence of i.i.d.\ random variables 
$Z_n=(Y_{2n-1},Y_{2n})$, $n \ge 1,$ so that each $Z_n$ is
uniformly distributed on $A\times A$.
From (\ref{eq:corr}), taking $k=2$ and an arbitrary even
$n=2\ell$,
\be
\Pp\big((X_1,X_2)=(a_1,a_2)\big|\hat{P}_{X_1^{2\ell}}=Q_{2\ell}\big)
=
	\Pp\big(Z_1=(a_1,a_2)\big|\hat{P}_{Z_1^\ell} \in E(Q_{2\ell})\big),
\label{eq:step3.1}
\ee
where $E(Q)$ denotes the set of probability
distributions $W$ on $A\times A$ with the property that
the average of the two marginals $W_1$ and $W_2$ of $W$
equals $Q$, 
\ben
E(Q) =\Big\{W\; \mbox{on}\; A\times A\,:\, \frac{W_1+W_2}{2} =Q\Big\}.
\een
If we write $U$ for the uniform distribution on $A\times A$,
it is easy to check that the distribution 
$W_\ell^*$ that uniquely achieves
the $\min_{W\in E(Q)} D(W\|U)$ is simply $Q\times Q$.
At this point, we would wish to apply the conditional limit 
theorem~\cite{cover:book2} to the i.i.d.\ process $\{Z_n\}$,
to obtain that,
\begin{eqnarray*}
\lim_{\ell\to\infty}
\Pp\big(Z_1=(a_1,a_2)\big|\hat{P}_{Z_1^\ell} \in E(Q_{2\ell})\big)
& = & 
	\lim_{\ell \to \infty}W_{\ell}^*(a_1,a_2) \\
& = & 
	\lim_{\ell \to \infty}Q_{2\ell}(a_1)Q_{2\ell}(a_2) \\
& = & 
	Q(a_1)Q(a_2),
\end{eqnarray*}
and combining this with (\ref{eq:step3.1}) would yield:
$$\Pp\big((X_1,X_2)=(a_1,a_2)\big|\hat{P}_{X_1^{2\ell}}=Q_{2\ell}\big)
\to Q(a_1)Q(a_2),\;\;\; \ell\to\infty.$$
The same argument can be used without difficulty 
to show that
for any $k\geq 1$ and any $a_1^k\in A^k$,
\be
\lim_{\ell \to \infty} 
\Pp(X_1^k=a_1^k|\hat{P}_{X_1^{k\ell}}=Q_{k\ell})= Q^k(a_1^k).
\label{eq:step3.2}
\ee

{\em Step 4}: 
Since (\ref{eq:step3.2}) holds for almost 
every sequence $\{Q_n\}$, letting $\ell\to\infty$, by the 
bounded converge theorem we have,
$$\Pp(X_1^k=a_1^k)  =  {\mathbb E}
\Big(\Pp(X_1^k=a_1^k|\hat{P}_{X_1^{k\ell}})\Big)
\to  \int_{\cal P} Q^k(a_1^k)d\bar{\mu}(Q),$$
as required.
\hfill$\Box$

\medskip

The only problem with the above argument
is that the set $E(Q)$ has an empty 
interior so that the conditional limit theorem is not directly applicable. 
Nevertheless, in the next section where we take a finite-$n$ approach,
we are able to `imitate' the proof of the conditional limit theorem 
and replace the step where the non-empty interior assumption is used 
with a different argument.

\subsection{Proof}

Recall the notation and terminology for types
described in the Introduction.
Let $\mu=\mu_n$ denote
the law of $\hat{P}_{X_1^n}$ on ${\cal P}$, and let
$\{Y_n\}$ be i.i.d.\ random variables 
uniformly distributed on $A$. For any $k\leq n$, 
any $n$-type $Q_n$, and any $a_1^k\in A^k$,
\ben
\Pp(X_1^k =a_1^k|\hat{P}_{X_1^n}  = Q_n)
  =  \Pp(Y_1^k =a_1^k|\hat{P}_{Y_1^n} = Q_n).
\een
Foe $k=1$ and any $a\in A$, by symmetry we have,
\ben
\Pp(X_1=a|\hat{P}_{X_1^n})
\;=\; \hat{P}_{X_1^n}(a),
\een
and taking the expectation of both sides with
respect to $\mu$ shows that in fact $P_{1}=M_{\mu,1}$.

For general $1\leq k\leq n$ with $n=k\ell$, 
for any $n$-type $Q$ we have,
\ben
\Pp(X^k_1=a_1^k | \hat{P}_{X_1^{k\ell}}=Q)  
& = & \Pp(Z_1=a_1^k | \hat{P}_{Z_1^\ell} \in E_{k}(Q))\\
& = & {\mathbb E}\big (\hat{P}_{Z_1^\ell}(a_1^k) \big| 
	\hat{P}_{Z_1^\ell} \in E_{k}(Q)\big),
\een
where now $\{Z_n\}$ is a sequence of i.i.d.\
random variables uniformly distributed on $A^k$,
and $E_{k}(Q)$ consists of all probability
distributions $W$ on $A^k$ with the property
that the average of the $k$ one-dimensional
marginals of $W$ equals $Q$.
Taking expectations with respect to 
$\hat{P}_{X_1^{k\ell}}=
\hat{P}_{X_1^{n}}\sim\mu=\mu_n$,
\ben
\Pp(X_1^k=a_1^k)
& = & 
	\int {\mathbb E}\big(\hat{P}_{Z_1^\ell}(a_1^k) \big| 
	\hat{P}_{Z_1^\ell} \in E_{k}(Q)\big)
		d\mu(Q),
\een
and by the joint convexity of relative entropy,
\be
D(P_k\|M_{\mu,k})
&=&
	D\Big(
	\int {\mathbb E}\big(\hat{P}_{Z_1^\ell}\big| 
	\hat{P}_{Z_1^\ell} \in E_{k}(Q)\big)
		d\mu(Q)
	\Big\|
	\int Q^k\;
		d\mu(Q)
	\Big)\nonumber\\
&\leq&
	\int
	D\Big(
	{\mathbb E}\big(\hat{P}_{Z_1^\ell}\big| 
	\hat{P}_{Z_1^\ell} \in E_{k}(Q)\big)
	\Big\|
	Q^k
	\Big)
	d\mu(Q)
	\nonumber\\
&\leq&
	\int
	{\mathbb E}\Big(
	D(\hat{P}_{Z_1^\ell}\|Q^k)\,
	\Big| 
	\hat{P}_{Z_1^\ell} \in E_{k}(Q)
	\Big)
	d\mu(Q).
	\label{eq:fromConv}
\ee

We will obtain an explicit bound for the relative entropy
in~(\ref{eq:fromConv}). First, we construct a joint
$\ell$-type $W$ with desirable properties.
Let ${\cal P}_\ell$ denote
the set of $\ell$-types on $A^k$.

\begin{lemma}
\label{lem:1}
For any $\ell> k\geq 1$ and
any $n$-type $Q$, there is a $W\in E_{k}(Q) \cap \mathcal{P}_\ell$ with:
$$\max_{a_1^k}|W(a_1^k)-Q^k(a_1^k)|\leq
	M:=
	\Bigg[\frac{2}{\ell} 
	+ \frac{4k}{\ell} + 2\sqrt{ \frac{k}{\ell} }\Bigg]^{1/2}.$$
Moreover, for $2\leq k\leq \sqrt{\ell}/10,$
$$|H(W)-H(Q^k)|\leq -M\log\frac{M}{m^k}.$$
\end{lemma}

\begin{proof}
Let $x_1^{k\ell}\in A^{k\ell}$ have type $Q$,
let $V_1^{k\ell}$ be a random permutation of $x_1^{k\ell}$,
and let $\What$ denote its (random) $\ell$-type.
Obviously we have that 
$\What\in E_{k}(Q)$ by construction,
and we will also show that $\What$ satisfies the
statement of the lemma with positive 
probability.
Taking any $k\leq \ell$ and $\gamma>0$ arbitrary, 
\begin{align}
\Pp\Big(\max_{a_1^k}
&\,
	|\What(a_1^k)-Q^k(a_1^k)|>\gamma\Big)
	\nonumber\\
&\leq
	\sum_{a_1^k}
	\Pp\Big( |\What(a_1^k)-Q^k(a_1^k)|>\gamma\Big)
	\nonumber\\
&\leq
	\sum_{a_1^k}
	\gamma^{-2}
	\mathbb{E}\Big[\Big(\What(a_1^k)-Q^k(a_1^k)\Big)^2\Big]
	\nonumber\\
&=
	\gamma^{-2}
	\sum_{a_1^k}
	\Big[\rho_2(a_1^k)
	-2Q^k(a_1^k)\rho_1(a_1^k)
	+Q^k(a_1^k)^2\Big],
	\label{eq:secondMom}
\end{align}
where
$\rho_2(a_1^k)=\mathbb{E}\big[\What(a_1^k)^2\big]$ and
$\rho_1(a_1^k)= \mathbb{E}\big[\What(a_1^k)\big]$.

Now we find appropriate bounds 
so that the above probability is $<1$. To get an 
upper bound on $\rho_1(a_1^k)$ for some fixed $a_1^k$
note that,
$$\rho_1(a_1^k) = \Pp(V_1^k=a_1^k)
\leq \prod_{i=1}^k\frac{n(a_i)}{\ell k-i+1},
$$
where $n(a_i)$ is the number of appearances of $a_i$
in $x_1^{k\ell}$, and hence,
\ben
\rho_1(a_1^k)
&\leq&
	\prod_{i=1}^k\frac{n(a_i)}{\ell k}\frac{\ell k}{\ell k-i+1}\\
&\leq&
	Q^k(a_1^k)\Big(\frac{\ell k}{\ell k-k}\Big)^k\\
&\leq&
	Q^k(a_1^k)\Big(1-\frac{1}{\ell}\Big)^{-k}\\
&\leq&
	Q^k(a_1^k)(1+k/\ell),
\een
since $(1-x)^{-k}\leq 1+kx$ for $x\in[0,1)$.
Similarly, writing $a_1^k*a_1^k$ for the concatenation
of $a_1^k$ with itself, we can estimate,
$$\Pp(V_1^{2k}=a_1^k*a_1^k)\leq Q^k(a_1^k)^2(1+4k/\ell),$$
so that,
\be
\rho_2(a_1^k)
&=&
	\ell\frac{1}{\ell^2}\rho_1(a_1^k)+\ell(\ell-1)\frac{1}{\ell^2}
	\Pp(V_1^{2k}=a_1^k*a_1^k)
	\nonumber\\
&\leq&
	(1/\ell)(1+k/\ell)Q^k(a_1^k)+(1+4k/\ell)Q^k(a_1^k)^2
	\nonumber\\
&\leq&
	(2/\ell)Q^k(a_1^k)+(1+4k/\ell)Q^k(a_1^k)^2.
	\label{eq:rho2bound}
\ee
Substituting the bound
(\ref{eq:rho2bound}) in (\ref{eq:secondMom}) we have,
\be
&&
\Pp\Big( \max_{a_1^k}|W(a_1^k)-Q^k(a_1^k)|>\gamma\Big)
	\nonumber\\
&&\leq
	\gamma^{-2}
	\sum_{a_1^k}
	Q^k(a_1^k)
	\Big[\frac{2}{\ell}+\Big(2+\frac{4k}{\ell}\Big)Q^k(a_1^k)
	-2\rho_1(a_1^k)
	\Big]
	\nonumber\\
&&\leq
	\gamma^{-2}
	\max_{a_1^k}
	\Big[\frac{2}{\ell}+\Big(2+\frac{4k}{\ell}\Big)Q^k(a_1^k)
	-2\rho_1(a_1^k)
	\Big].
	\label{eq:norho1}
\ee
Finally, we get a lower bound on 
$\rho_1(a_1^k)$.
In the case where for some $\beta>0$
(to be chosen later),
$Q(a_i)> \beta$ for all $a_i$, 
we have,
$$
\rho_1(a_1^k)\geq
	\prod_{i=1}^k\frac{n(a_i)-i+1}{\ell k-i+1}\geq
	\prod_{i=1}^k\frac{n(a_i)-k+1}{\ell k}\geq
	Q^k(a_1^k)\prod_{i=1}^k\Big(1-\frac{1}{\ell Q(a_i)}\Big),
$$
so that,
\be
\rho_1(a_1^k)
>
	Q^k(a_1^k)\Big(1-\frac{1}{\ell\beta}\Big)^k
\geq
	Q^k(a_1^k)\Big(1-\frac{k}{\ell\beta}\Big),
	\label{eq:rho1bound}
\ee
assuming $\beta\geq 1/2\ell,$ since $(1-x)^k\geq 1-kx$ for
all $k\geq 1$ and all $x\in[0,2]$.
For all such $a_1^k$, 
using (\ref{eq:rho1bound}) 
we can bound the expression 
in the maximum in (\ref{eq:norho1}) by
$$\Big[\frac{2}{\ell}+\Big(2+\frac{4k}{\ell}\Big)Q^k(a_1^k)
	-2\rho_1(a_1^k)
	\Big]
	<
	\frac{2}{\ell} +\frac{4k}{\ell} +\frac{2k}{\ell\beta}.$$
And in the case when at least one $a_i$ has $Q(a_i)\leq\beta$,
simply omitting the negative term and noting that $Q^k(a_1^k)\leq\beta$
we bound the same term above by,
$$\Big[\frac{2}{\ell}+\Big(2+\frac{4k}{\ell}\Big)Q^k(a_1^k)
	\Big]
	\leq
	\frac{2}{\ell} +
	\Big(2+\frac{4k}{\ell}\Big)\beta.
$$
Combining the last three bounds,
$$\Pp\Big( \max_{a_1^k}|W(a_1^k)-Q^k(a_1^k)|>\gamma\Big)
\leq
	\gamma^{-2}
	\Big[\frac{2}{\ell}
	+
	\max
	\Big\{
	\frac{2k}{\ell}\big(2+\frac{1}{\beta}\big),\;
	\big(2+\frac{4k}{\ell}\big)\beta
	\Big\}\Big],
$$
where the inequality is strict when the first
term dominates the maximum.
To obtain a good bound we
take for $\beta$ a value approximately equal to
the minimiser of the above expression: We
set $\beta^*=\sqrt{k/\ell}$.
Note that for this $\beta^*$ it can be easily 
verified that the first term strictly dominates 
the maximum, giving,
$$\Pp\Big( \max_{a_1^k}|W(a_1^k)-Q^k(a_1^k)|>\gamma\Big)
<
	\gamma^{-2}
	\Big[\frac{2}{\ell} + \frac{4k}{\ell} + 2\sqrt{ \frac{k}{\ell} }\Big],
$$
and taking $\gamma=M$ as in the lemma,
completes the proof of the first statement. 

For the second part, noting that for $2\leq k\leq\sqrt{\ell}/10$
we have $M<1/2$, the result follows 
from~\cite[Lemma~2.7]{csiszar:book}.
\end{proof}

Next we obtain an upper bound on the conditional
expectation 
in~(\ref{eq:fromConv}). 

\begin{lemma}\label{Dbound}
Suppose $n = \ell k$, with $2\leq k\leq\sqrt{\ell}/10$. 
For any $n$-type $Q$
we have:
$$	\mathbb{E}\Big(
	D(\hat{P}_{Z_1^\ell}\|Q^k)\,
	\Big| 
	\hat{P}_{Z_1^\ell} \in E_{k}(Q)
	\Big)
\leq \epsilon(n,k).
$$
\end{lemma}

\begin{proof}
We follow the same steps as in the proof
of the conditional limit theorem in~\cite{cover:book2}.
Recall that if we write $U_k$ for the uniform distribution 
on $A^k$, then the $W_{k}^*$ that uniquely achieves
$D^*=\min_{W\in E_{k}(Q)} D(W\|U_k)$ is $W_{k}^*=Q^k$.
We partition $E_{k}(Q)$ into $B_{2 \delta}$ and $C=E_{k}(Q)-B_{2 \delta}$, 
where $B_{2 \delta}=\{W \in E_k(Q) \,:\,D(W\|U_k) \le D^*+2 \delta \}$, with $\delta = \delta_{n,k}$.
Then, writing $\nu_\ell$ for the distribution of
$\hat{P}_{Z_1^\ell}$,
\begin{eqnarray*}
  \nu_\ell(C|E_{k}(Q))
	= \frac {\nu_\ell(C\cap E_{k}(Q))}{\nu_\ell(E_{k}(Q))} 
	\le \frac {\nu_\ell(C)} {\nu_\ell(B_{2 \delta})}.
\end{eqnarray*}
Next we bound the above numerator and denominator.
For the numerator, writing again
${\cal P}_\ell$ for the set of $\ell$-types on $A^k$,
\begin{eqnarray*}
\nu_\ell(C)
& \eqa & \sum_{W \in C\cap \mathcal{P}_\ell} U_k^\ell(T(W))\\
& \leqb & \sum_{W \in C\cap \mathcal{P}_\ell} e^{-\ell D(W\|U_k)}\\
& \leqc & |E_{k}(Q)\cap \mathcal{P}_\ell|e^{-\ell(D^*+2 \delta)}\\
& \leqd & (\ell+1)^{m^{k}}e^{-\ell(D^*+2 \delta)}
\end{eqnarray*}
where $T(W)$ in $(a)$ denotes the type class of all
strings of length $\ell$ in $A^k$ with type $W$,
$(b)$ is a standard property~\cite{cover:book2}, $(c)$ follows from the 
definition of $C$ and the fact that 
$E_{k}(Q)\cap \mathcal{P}_\ell \subset E_k(Q)$, 
and $(d)$ follows from the 
standard observation that 
$|E_{k}(Q)\cap \mathcal{P}_\ell|\leq | \mathcal{P}_\ell| 
\leq (\ell+1)^{m^{k}}$.
Similarly, letting $W_0$ denote the type from Lemma~\ref{lem:1},
\begin{eqnarray*}
\nu_\ell(B_{2 \delta})
& \ge & 
	\nu_\ell(B_{\delta})\\
& = & \sum_{W \in B_{\delta}\cap \mathcal{P}_\ell} U_k^\ell(T(W))\\
& \ge & U_k^\ell(T(W_0))\\
& \ge & (\ell+1)^{-m^k} e^{-\ell D(W_0\|U)}\\
& \ge & (\ell+1)^{-m^k} e^{-\ell(D^*+\delta)}.
\end{eqnarray*}
Combining these bounds, we obtain,
$$\Pp\big(\hat{P}_{Z_1^\ell} \in C\big|\hat{P}_{Z_1^\ell} \in E_k(Q)\big) 
\le (\ell+1)^{2m^k}e^{-\ell \delta},$$
or,
\begin{eqnarray*}
\Pp\Bigl(D(\hat{P}_{Z_1^\ell}\|U_k)>D^*+2 \delta
\Big|\hat{P}_{Z_1^\ell} \in E_k(Q)\Bigr) 
\le (\ell+1)^{2m^k}e^{-\ell \delta}.
\end{eqnarray*}
Since the set $E_k(Q)$ is closed and convex, we may apply the 
Pythagorean identity for relative 
entropy~\cite{cover:book2} to conclude that:
\begin{eqnarray*}
\Pp\Bigl(D(\hat{P}_{Z_1^\ell}\|Q^k)>2 \delta |\hat{P}_{Z_1^\ell} 
\in E_k(Q)\Bigr) 
\le (\ell+1)^{2m^k} e^{-\ell \delta}.
\end{eqnarray*}
Thus, 
\begin{equation*}
\E{D(\hat{P}_{Z_1^\ell}\|Q^k)\big|\hat{P}_{Z_1^\ell} \in E_\ell(Q)}
\le  (\ell+1)^{2m^k} e^{-\ell \delta}\max_{P \in E_k(Q)} {D(P\|Q^k)}+2 \delta.
\end{equation*}
The claimed bound now follows by Lemma~\ref{lemmamaxD} 
on taking $\ell = k/n$.
\end{proof}

\begin{lemma} \label{lemmamaxD}
For any $n$-type $Q$, $\max_{W \in E_k(Q)} D(W\|Q^k) \leq k\log{n}$.
\end{lemma}

\begin{proof}
If  $a_1^k \in A^k$ is such that $Q^k(a_1^k) = \prod_{i=1}^k{Q(a_i)}=0$,
then $Q(a_{i_0}) = 0$ for some $i_0$. Since 
$Q(a_{i_0}) = \frac{1}{k}\sum_{j=1}^k{W_j(a_{i_0})}$, we must 
have $W_1(a_{i_0}) = \cdots = W_k(a_{i_0}) = 0$,
which implies that $W(a_1^k)= 0.$ 

On the other hand, if $Q^k(a_1^k)>0$ then $Q^k(a_1^k) \ge \frac {1}{n^k}$.  
Thus, for any $W \in E_k(Q)$, 
\begin{eqnarray*}
D(W\|Q^k)
& = & \sum_{a_1^k \in A^k}W(a_1^k)\log \frac {W(a_1^k)} {Q^k(a_1^k)}\\
& \le & \sum_{a_1^k \in A^k}W(a_1^k)\log \frac {W(a_1^k)}  {({\frac {1} {n}})^k}\\
& = & k\log n- H(W)\\
& \le & k\log n,
\end{eqnarray*}
as required.
\end{proof}

Theorem~\ref{thm} follows from~\eqref{eq:fromConv} 
combined with Lemma~\ref{Dbound}.

\begin{acknowledgement}
We wish to thank Andrew Barron for pointing out several useful
references for the historical development outlined in our Introduction.
\end{acknowledgement}

\bibliographystyle{plain}

\begin{thebibliography}{10}

\bibitem{algoet:92}
P.H. Algoet and B.H. Marcus.
\newblock Large deviation theorems for empirical types of {Markov} chains
  constrained to thin sets.
\newblock {\em IEEE Trans. Inform. Theory}, 38(4):1276--1291, July 1992.

\bibitem{artstein:04}
S.~Artstein, K.~Ball, F.~Barthe, and A.~Naor.
\newblock Solution of {Shannon's} problem on the monotonicity of entropy.
\newblock {\em J. Amer. Math. Soc.}, 17(4):975--982, 2004.

\bibitem{bakry-emery:85}
D.~Bakry and M.~{\'E}mery.
\newblock Diffusions hypercontractives.
\newblock In {\em Seminaire de probabilit{\'e}s XIX 1983/84}, pages 177--206.
  Springer, 1985.

\bibitem{Kcompound:10}
A.D. Barbour, O.~Johnson, I.~Kontoyiannis, and M.~Madiman.
\newblock Compound {Poisson} approximation via information functionals.
\newblock {\em Electron. J. Probab}, 15:1344--1369, 2010.

\bibitem{barron:clt}
A.R. Barron.
\newblock Entropy and the central limit theorem.
\newblock {\em Ann. Probab.}, 14(1):336--342, January 1986.

\bibitem{barron:97}
A.R. Barron.
\newblock Information theory in probability, statistics, learning, and neural
  nets.
\newblock In Y.~Freund and R.E. Schapire, editors, {\em Proceedings of the
  Tenth Annual Conference on Computational Learning Theory (COLT)}, Nashville,
  Tennessee, July 1997.
\newblock Available at:
  http://www.stat.yale.edu/$\sim$arb4/publications\_files/COLT97.pdf.

\bibitem{barron:isit00}
A.R. Barron.
\newblock Limits of information, {M}arkov chains, and projection.
\newblock In {\em 2000 IEEE International Symposium on Information Theory
  (ISIT)}, Sorrento, Italy, June 2000.

\bibitem{bayarri:04}
M.~Bayarri and J.O. Berger.
\newblock The interplay of {Bayesian} and frequentist analysis.
\newblock {\em Statistical Science}, 19(1):58--80, 2004.

\bibitem{blachman:65}
N.M. Blachman.
\newblock The convolution inequality for entropy powers.
\newblock {\em IEEE Trans. Inform. Theory}, 11(2):267--271, April 1965.

\bibitem{bobkov:99}
S.G. Bobkov and F.~G{\"o}tze.
\newblock Exponential integrability and transportation cost related to
  logarithmic {S}obolev inequalities.
\newblock {\em J. Funct. Anal.}, 163(1):1--28, 1999.

\bibitem{bobkov-ledoux:98}
S.G. Bobkov and M.~Ledoux.
\newblock On modified logarithmic {S}obolev inequalities for {B}ernoulli and
  {P}oisson measures.
\newblock {\em J. Funct. Anal.}, 156(2):347--365, 1998.

\bibitem{boucheron:book}
S.~Boucheron, G.~Lugosi, and P.~Massart.
\newblock {\em Concentration inequalities: {A} nonasymptotic theory of
  independence}.
\newblock Oxford University Press, Oxford, U.K., 2013.

\bibitem{brown:82}
L.D. Brown.
\newblock A proof of the central limit theorem motivated by the
  {C}ram\'er-{R}ao inequality.
\newblock In {\em Statistics and Probability: Essays in Honor of C.R. Rao},
  pages 141--148. North-Holland, Amsterdam, 1982.

\bibitem{brown:86}
L.D. Brown.
\newblock {\em Fundamentals of statistical exponential families: With
  applications in statistical decision theory}, volume~9 of {\em IMS Lecture
  Notes Monograph Series}.
\newblock Institute of Mathematical Statistics, Hayward, CA, 1986.

\bibitem{bubeck:18}
S.~Bubeck and S.~Ganguly.
\newblock Entropic {CLT} and phase transition in high-dimensional {W}ishart
  matrices.
\newblock {\em International Mathematics Research Notices}, 2018(2):588--606,
  2018.

\bibitem{carlen:09}
E.A Carlen and D.~Cordero-Erausquin.
\newblock Subadditivity of the entropy and its relation to {Brascamp-Lieb} type
  inequalities.
\newblock {\em Geometric \& Functional Analysis}, 19(2):373--405, 2009.

\bibitem{carlen:91}
E.A. Carlen and A.~Soffer.
\newblock Entropy production by block variable summation and central limit
  theorems.
\newblock {\em Comm. Math. Phys.}, 140(2):339--371, 1991.

\bibitem{costa-cover:84}
M.H.M. Costa and T.M. Cover.
\newblock On the similarity of the entropy power inequality and the
  {B}runn-{M}inkowski inequality.
\newblock {\em IEEE Trans. Inform. Theory}, 30(6):837--839, November 1984.

\bibitem{cover:book2}
T.M. Cover and J.A. Thomas.
\newblock {\em Elements of information theory}.
\newblock J. Wiley \& Sons, New York, second edition, 2012.

\bibitem{csiszar:63}
I.~Csisz{\'a}r.
\newblock Eine informationstheoretische ungleichung und ihre anwendung auf den
  beweis der ergodizität von {M}arkoffschen ketten.
\newblock {\em Publ. Math. Inst. Hungar. Acad. Sci.}, 8:85--108, 1963.

\bibitem{csiszar:67}
I.~Csisz{\'a}r.
\newblock Information-type measures of difference of probability distributions
  and indirect observations.
\newblock {\em Studia Sci. Math. Hungar.}, 2:299--318, 1967.

\bibitem{csiszar:75}
I.~Csisz{\'{a}}r.
\newblock ${I}$-divergence geometry of probability distributions and
  minimization problems.
\newblock {\em Ann. Probab.}, 3(1):146--158, February 1975.

\bibitem{csiszar:84}
I.~Csisz{\'a}r.
\newblock Sanov property, generalized ${I}$-projection and a conditional limit
  theorem.
\newblock {\em Ann. Probab.}, 12(3):768--793, August 1984.

\bibitem{csiszar:97}
I~Csisz{\'a}r.
\newblock Information theoretic methods in probability and statistics.
\newblock In {\em 1997 IEEE International Symposium on Information Theory
  (ISIT)}, Ulm, Germany, June 1997.

\bibitem{csiszar:98}
I.~Csisz{\'a}r.
\newblock The method of types.
\newblock {\em IEEE Trans. Inform. Theory}, 44(6):2505--2523, October 1998.
\newblock Information theory: 1948--1998.

\bibitem{csiszar:06}
I.~Csisz{\'a}r.
\newblock A simple proof of {Sanov}'s theorem.
\newblock {\em Bulletin of the Brazilian Mathematical Society}, 37(4), 2006.

\bibitem{csiszar-cover-choi:87}
I.~Csisz{\'a}r, T.M. Cover, and B.S. Choi.
\newblock Conditional limit theorems under {M}arkov conditioning.
\newblock {\em IEEE Trans. Inform. Theory}, 33(6):788--801, November 1987.

\bibitem{csiszar:book}
I.~Csisz{\'{a}}r and J.~K{\"{o}}rner.
\newblock {\em Information theory: Coding theorems for discrete memoryless
  systems}.
\newblock Academic Press, New York, 1981.

\bibitem{davies-simon:84}
E.B. Davies and B.~Simon.
\newblock Ultracontractivity and the heat kernel for {S}chr\"odinger operators
  and {D}irichlet {L}aplacians.
\newblock {\em J. Funct. Anal.}, 59(2):335--395, 1984.

\bibitem{definetti:31}
B.~De~Finetti.
\newblock Sul significato soggettivo della probabilita.
\newblock {\em Fundamenta Mathematicae}, 17(1):298--329, 1931.

\bibitem{definetti:37}
B.~De~Finetti.
\newblock La pr{\'e}vision: ses lois logiques, ses sources subjectives.
\newblock {\em Ann. Inst. Henri Poincar\'e}, 7(1):1--68, 1937.

\bibitem{dembo-cover-thomas:91}
A.~Dembo, T.M. Cover, and J.A. Thomas.
\newblock Information-theoretic inequalities.
\newblock {\em IEEE Trans. Inform. Theory}, 37(6):1501--1518, November 1991.

\bibitem{diaconis:77}
P.~Diaconis.
\newblock Finite forms of de {F}inetti's theorem on exchangeability.
\newblock {\em Synthese}, 36(2):271--281, 1977.

\bibitem{diaconis-freedman:80b}
P.~Diaconis and D.A. Freedman.
\newblock Finite exchangeable sequences.
\newblock {\em Ann. Probab.}, 8(4):745--764, 1980.

\bibitem{fritz:73}
J.~Fritz.
\newblock An information-theoretical proof of limit theorems for reversible
  {M}arkov processes.
\newblock In {\em Transactions of the Sixth Prague Conference on Information
  Theory, Statistical Decision Functions, Random Processes (Tech. Univ.,
  Prague, 1971; dedicated to the memory of Anton\'\i n \v Spa\v cek)}, pages
  183--197. Academia, Prague, 1973.

\bibitem{gavalakis-arxiv:clt}
L.~Gavalakis and I.~Kontoyiannis.
\newblock Entropy and the discrete central limit theorem.
\newblock {\em arXiv e-prints}, \texttt{2106.00514 [math.PR]}, June 2021.

\bibitem{gavalakis:21}
L.~Gavalakis and I.~Kontoyiannis.
\newblock An information-theoretic proof of a finite de {F}inetti theorem.
\newblock {\em Electron. Comm. Probab.}, 26:1--5, 2021.

\bibitem{GOR:79}
P.~Groeneboom, J.~Oosterhoff, and F.H. Ruymgaart.
\newblock Large deviation theorems for empirical probability measures.
\newblock {\em Ann. Probab.}, 7(4):553--586, August 1979.

\bibitem{gross:75}
L.~Gross.
\newblock Logarithmic {S}obolev inequalities.
\newblock {\em Amer. J. Math.}, 97(4):1061--1083, 1975.

\bibitem{guo:05}
D.~Guo, S.~Shamai, and S.~Verd{\'u}.
\newblock Mutual information and minimum mean-square error in {Gaussian}
  channels.
\newblock {\em IEEE Trans. Inform. Theory}, 51(4):1261--1282, April 2005.

\bibitem{hajek:58a}
J.~H{\'a}jek.
\newblock On a property of normal distributions of any stochastic process.
\newblock {\em Czechoslovak Math. J.}, 8(83):610--618, 1958.

\bibitem{hajek:58b}
J.~H{\'a}jek.
\newblock A property of {J}-divergences of marginal probability distributions.
\newblock {\em Czechoslovak Math. J.}, 8(3):460--463, 1958.

\bibitem{harremoes:01}
P.~Harremo{\"e}s.
\newblock Binomial and {P}oisson distributions as maximum entropy
  distributions.
\newblock {\em IEEE Trans. Inform. Theory}, 47(5):2039--2041, July 2001.

\bibitem{harremoes:09}
P.~Harremo{\"e}s and K.K. Holst.
\newblock Convergence of {M}arkov chains in information divergence.
\newblock {\em J. Theoret. Probab.}, 22(1):186--202, March 2009.

\bibitem{harremoesJK:07}
P.~Harremo{\"e}s, O.~Johnson, and I.~Kontoyiannis.
\newblock Thinning and the law of small numbers.
\newblock In {\em 2007 IEEE International Symposium on Information Theory
  (ISIT)}, pages 1491--1495, Nice, France, June 2007.

\bibitem{konto-H-J:10}
P.~Harremo{\"e}s, O.~Johnson, and I.~Kontoyiannis.
\newblock Thinning, entropy, and the law of thin numbers.
\newblock {\em IEEE Trans. Inform. Theory}, 56(9):4228--4244, 2010.

\bibitem{hewitt-savage:55}
E.~Hewitt and L.J. Savage.
\newblock Symmetric measures on {C}artesian products.
\newblock {\em Trans. Amer. Math. Soc.}, 80(2):470--501, 1955.

\bibitem{johnson:00}
O.~Johnson.
\newblock Entropy inequalities and the central limit theorem.
\newblock {\em Stoch. Proc. Appl.}, 88(2):291--304, 2000.

\bibitem{johnson:book}
O.~Johnson.
\newblock {\em Information theory and the central limit theorem}.
\newblock Imperial College Press, London, U.K., 2004.

\bibitem{johnson:07}
O.~Johnson.
\newblock Log-concavity and the maximum entropy property of the {P}oisson
  distribution.
\newblock {\em Stoch. Proc. Appl.}, 117(6):791--802, June 2007.

\bibitem{johnson-barron:04}
O.~Johnson and A.R. Barron.
\newblock Fisher information inequalities and the central limit theorem.
\newblock {\em Probab. Theory Related Fields}, 129(3):391--409, 2004.

\bibitem{jKm:13}
O.~Johnson, I.~Kontoyiannis, and M.~Madiman.
\newblock Log-concavity, ultra-log-concavity, and a maximum entropy property of
  discrete compound {Poisson} measures.
\newblock {\em Discrete Applied Mathematics}, 161(9):1232--1250, 2013.

\bibitem{johnson:00b}
O.~Johnson and Yu.M. Suhov.
\newblock Entropy and convergence on compact groups.
\newblock {\em J. Theoret. Probab.}, 13(3):843--857, 2000.

\bibitem{kemperman:69}
J.H.B. Kemperman.
\newblock On the optimum rate of transmitting information.
\newblock In {\em Probability and information theory}, pages 126--169.
  Springer-Verlag, 1969.

\bibitem{kendall:63}
D.G. Kendall.
\newblock Information theory and the limit-theorem for {M}arkov chains and
  processes with a countable infinity of states.
\newblock {\em Ann. Inst. Statist. Math.}, 15(1):137--143, May 1963.

\bibitem{kolmogorov:58}
A.N. Kolmogorov.
\newblock A new metric invariant of transitive dynamical systems and {L}ebesgue
  space automorphisms.
\newblock In {\em Dokl. Acad. Sci. USSR}, volume 119, pages 861--864, 1958.

\bibitem{konto-H-J:05}
I.~Kontoyiannis, P.~Harremo{\"e}s, and O.~Johnson.
\newblock Entropy and the law of small numbers.
\newblock {\em IEEE Trans. Inform. Theory}, 51(2):466--472, February 2005.

\bibitem{KM:14}
I.~Kontoyiannis and M.~Madiman.
\newblock Sumset and inverse sumset inequalities for differential entropy and
  mutual information.
\newblock {\em IEEE Trans. Inform. Theory}, 60(8):4503--4514, August 2014.

\bibitem{kullback:67}
S.~Kullback.
\newblock A lower bound for discrimination information in terms of variation.
\newblock {\em IEEE Trans. Inform. Theory}, 13(1):126--127, January 1967.

\bibitem{ledoux:96}
M.~Ledoux.
\newblock Isoperimetry and {G}aussian analysis.
\newblock In {\em Lectures on probability theory and statistics}, volume 1648
  of {\em Lecture Notes in Mathematics}, pages 165--294. Springer, Berlin,
  1996.

\bibitem{ledoux:97}
M.~Ledoux.
\newblock On {T}alagrand's deviation inequalities for product measures.
\newblock {\em ESAIM Probab. Statist.}, 1:63--87 (electronic), 1997.

\bibitem{ledoux:book}
M.~Ledoux.
\newblock {\em The concentration of measure phenomenon}.
\newblock American Mathematical Society, Providence, RI, 2001.

\bibitem{linnik:59}
Ju.V. Linnik.
\newblock An information-theoretic proof of the central limit theorem with
  {L}indeberg conditions.
\newblock {\em Theory Probab. Appl.}, 4:288--299, 1959.

\bibitem{courtade:16}
J.~Liu, T.A. Courtade, P.~Cuff, and S.~Verd{\'u}.
\newblock Brascamp-{L}ieb inequality and its reverse: {A}n information
  theoretic view.
\newblock In {\em 2016 IEEE International Symposium on Information Theory
  (ISIT)}, pages 1048--1052, Barcelona, Spain, July 2016.

\bibitem{madiman:07}
M.~Madiman and A.R Barron.
\newblock Generalized entropy power inequalities and monotonicity properties of
  information.
\newblock {\em IEEE Trans. Inform. Theory}, 53(7):2317--2329, July 2007.

\bibitem{madimanJK:07}
M.~Madiman, O.~Johnson, and I.~Kontoyiannis.
\newblock Fisher information, compound {Poisson} approximation, and the
  {Poisson} channel.
\newblock In {\em 2007 IEEE International Symposium on Information Theory
  (ISIT)}, pages 976--980, Nice, France, June 2007.

\bibitem{KM:16}
M.~Madiman and I.~Kontoyiannis.
\newblock Entropy bounds on abelian groups and the {Ruzsa} divergence.
\newblock {\em IEEE Trans. Inform. Theory}, 64(1):77--92, January 2016.

\bibitem{marton:96}
K.~Marton.
\newblock Bounding {$\overline d$}-distance by informational divergence: {A}
  method to prove measure concentration.
\newblock {\em Ann. Probab.}, 24(2):857--866, April 1996.

\bibitem{oconnell:TR}
N.~O'Connell.
\newblock Information-theoretic proof of the {H}ewitt-{S}avage zero-one law.
\newblock Technical report, Hewlett-Packard Laboratories, Bristol, U.K., June
  2000.

\bibitem{ornstein:70a}
D.S. Ornstein.
\newblock {B}ernoulli shifts with the same entropy are isomorphic.
\newblock {\em Advances in Mathematics}, 4(3):337--352, 1970.

\bibitem{ornstein:70c}
D.S. Ornstein.
\newblock Imbedding {B}ernoulli shifts in flows.
\newblock In {\em Contributions to ergodic theory and probability}, pages
  178--218. Springer, 1970.

\bibitem{ornstein:70b}
D.S. Ornstein.
\newblock Two {B}ernoulli shifts with infinite entropy are isomorphic.
\newblock {\em Advances in Mathematics}, 5(3):339--348, 1970.

\bibitem{raginsky:book}
M.~Raginsky and I.~Sason.
\newblock {\em Concentration of Measure Inequalities in Information Theory,
  Communications and Coding}.
\newblock Foundations and Trends in Communications and Information Theory. NOW
  Publishers, Boston, MA, 2018.

\bibitem{renyi:61}
A.~R{\'e}nyi.
\newblock On measures of entropy and information.
\newblock In {\em Proc. 4th Berkeley Sympos. Math. Statist. and Prob., Vol. I},
  pages 547--561. Univ. California Press, Berkeley, Calif., 1961.

\bibitem{sanov:57}
I.N. Sanov.
\newblock On the probability of large deviations of random variables.
\newblock {\em Mat. Sb.}, 42:11--44, 1957.
\newblock English translation in {\em Sel. Transl. Math. Statist. Probab.}
  (1961) {\bf 1}, 213-244.

\bibitem{shannon:48}
C.E. Shannon.
\newblock A mathematical theory of communication.
\newblock {\em Bell System Tech. J.}, 27(3):379--423, 623--656, 1948.

\bibitem{shannon:56}
C.E Shannon.
\newblock The bandwagon.
\newblock {\em IRE Trans. Inform. Theory}, 2(1):3, March 1956.

\bibitem{shimizu:75}
R.~Shimizu.
\newblock On {Fisher's} amount of information for location family.
\newblock In G.P. Patil, S.~Kotz, and J.K. Ord, editors, {\em A Modern Course
  on Statistical Distributions in Scientific Work}, pages 305--312. Springer,
  Dordrecth, Netherlands, 1975.

\bibitem{stam:59}
A.J. Stam.
\newblock Some inequalities satisfied by the quantities of information of
  {F}isher and {S}hannon.
\newblock {\em Information and Control}, 2(2):101--112, 1959.

\bibitem{tao:10}
T.~Tao.
\newblock Sumset and inverse sumset theory for {Shannon} entropy.
\newblock {\em Combinatorics, Probability and Computing}, 19(4):603--639, July
  2010.

\bibitem{topsoe:79}
F.~Tops{\o}e.
\newblock Information-theoretical optimization techniques.
\newblock {\em Kybernetika}, 15(1):8--27, 1979.

\bibitem{tulino:06}
A.M. Tulino and S.~Verd{\'u}.
\newblock Monotonic decrease of the non-{Gaussianness} of the sum of
  independent random variables: {A} simple proof.
\newblock {\em IEEE Trans. Inform. Theory}, 52(9):4295--4297, September 2006.

\bibitem{villani:book}
C.~Villani.
\newblock {\em Optimal transport: Old and new}.
\newblock Springer, Berlin, 2009.

\bibitem{voiculescu:I}
D.~Voiculescu.
\newblock The analogues of entropy and of {F}isher's information measure in
  free probability theory, {I}.
\newblock {\em Comm. Math. Phys.}, 155(1):71--92, 1993.

\bibitem{voiculescu:II}
D.~Voiculescu.
\newblock The analogues of entropy and of {F}isher's information measure in
  free probability theory, {II}.
\newblock {\em Inventiones Mathematicae}, 118(1):411--440, 1994.

\bibitem{voiculescu:III}
D.~Voiculescu.
\newblock The analogues of entropy and of {Fisher's} information measure in
  free probability theory {III}: {The} absence of {Cartan} subalgebras.
\newblock {\em Geometric \& Functional Analysis}, 6(1):172--199, 1996.

\bibitem{voiculescu:IV}
D.~Voiculescu.
\newblock The analogues of entropy and of {F}isher's information measure in
  free probability theory, {IV}: {M}aximum entropy and freeness, in free
  probability theory.
\newblock {\em Fields Inst. Commun.}, 12:293--302, 1997.

\bibitem{wolfowitz:61book}
J.~Wolfowitz.
\newblock {\em Coding theorems of information theory}.
\newblock Springer, Berlin, 1961.

\bibitem{yu:09}
Y.~Yu.
\newblock On the entropy of compound distributions on nonnegative integers.
\newblock {\em IEEE Trans. Inform. Theory}, 55(8):3645--3650, August 2009.

\end{thebibliography}

\def\cprime{$'$}

\end{document}